\documentclass[a4paper,10pt]{article}
\usepackage[truedimen,margin=20mm]{geometry} 

\usepackage{amsmath}
\usepackage{amsthm}
\usepackage{amssymb}
\usepackage{graphicx}
\usepackage{color}
\usepackage{lmodern}
\usepackage{kotex}
\usepackage{caption}
\usepackage{subcaption}
\usepackage{nicematrix}
\usepackage{booktabs}
\usepackage{enumitem}
\usepackage{blkarray}
\usepackage{array}

\usepackage[
defernumbers=true,
backend=biber,
style=numeric,sortcites,
sorting=nty,natbib,
maxbibnames=99]{biblatex}
\renewbibmacro{in:}{} 
\usepackage{booktabs}
\usepackage{tikz}

\usepackage[ruled,vlined]{algorithm2e}
\SetAlFnt{\small}
\SetKwInput{KwInput}{Input}
\SetKwInput{KwOutput}{Output}
\NoCaptionOfAlgo


\theoremstyle{plain}
\newtheorem{thm}{Theorem}[section]
\newtheorem{lem}[thm]{Lemma}

\newtheorem{prop}[thm]{Proposition}

\theoremstyle{definition}

\newcommand{\Delete}[1]{}
\newcommand\scalemath[2]{\scalebox{#1}{\mbox{\ensuremath{\displaystyle #2}}}}

\newcommand\blfootnote[1]{%
  \begingroup
  \renewcommand\thefootnote{}\footnote{#1}%
  \addtocounter{footnote}{-1}%
  \endgroup
}

\title{On non-bipartite graphs with integral signless Laplacian eigenvalues at most $6$}

\author{
  {\sc Semin OH}
  \footnote{KNU G-LAMP Project Group, KNU Institute of Basic Sciences, Kyungpook National University, Daegu, 41566, Republic of Korea. \textit{Email:} \texttt{semin@knu.ac.kr}}
  \blfootnote{All authors contributed equally to this work.}
  \and
  {\sc Jeong Rye PARK}
  \footnote{Department of Mathematics, Kyungpook National University, Daegu, 41566, Republic of Korea. \textit{Email:} \texttt{parkjr@knu.ac.kr}}
  \and
  {\sc Jongyook PARK}
  \footnote{Department of Mathematics, Kyungpook National University, Daegu, 41566, Republic of Korea. \textit{Email:} \texttt{jongyook@knu.ac.kr}}
  \and
  {\sc Yoshio SANO}
  \footnote{Division of Information Engineering, Faculty of Engineering, Information and Systems, University of Tsukuba, Tsukuba, Ibaraki, 305-8537, Japan. \textit{Email:} \texttt{sano@cs.tsukuba.ac.jp}} 
} 

\date{}

\begin{document}

\maketitle


\begin{abstract}

In this paper, 
we completely classify the connected non-bipartite graphs 
with integral signless Laplacian eigenvalues at most $6$.



\end{abstract}




\section{Introduction}
All graphs under consideration in this paper are finite, connected, and simple.
Definitions and notations not introduced are provided in the next section.

The {\it spectrum} of a graph $G$ is, by definition, the spectrum of its adjacency matrix $A(G)$: the set of the eigenvalues of $G$ together with their multiplicities. 
There are several literatures on the graph spectrum; see~\cite {CDS80, CRS10, BH11}.

In 1974, integral graphs were introduced by Harary and Schwenk~\cite{HS}.
A graph $G$ is called an \emph{integral graph} if it has only integral eigenvalues. For a survey on integral graphs, see \cite{BCRSS02}.
Integral graphs exist in unlimited numbers across all classes of graphs and among graphs of all orders.
There are some results on some particular classes of integral graphs:
trees~\cite{WS79, W79, HN98};  cubic graphs~\cite{BC76, Cve75, Sch78}.
There are some results on integral graphs with small vertex degrees: maximum vertex degree 3~\cite{CGT74, BC76, Cve75, Sch78}; maximum vertex degree 4~\cite{BS01a, BS01b, RS86, Lep05, RS95}.
The case of integral graphs with maximum vertex degree 3 is completely classified, but the classification of integral graphs with maximum vertex degree 4 is still open.
In general, integral graphs are scarce and hard to find.

Let $R$ be the $n \times m$ vertex-edge incidence matrix of a graph $G$. 
Denote the line graph of $G$ by $L_G$. 
The following relations are well known~\cite{CRS2}:
\begin{equation} \label{eq:R}
  RR^T =A(G) +D(G) \textrm{~and~} R^TR=A(L_G) +2I_{m},  
\end{equation}
where $D(G)$ is the diagonal matrix of vertex-degrees in $G$
and $I_{m}$ is the identity matrix of order $m$. 
The characteristic polynomial $P_G(\lambda)=\det(\lambda I_{n}-A(G))$ is called the {\it characteristic polynomial} of $G$.
From~(\ref{eq:R}), it immediately follows that
\[P_{L_G}(\lambda) = (\lambda + 2)^{m-n}\textit{det}((\lambda+2) I_{n} - Q(G)),\]
where $Q(G)=A(G)+D(G)$. 
It means that $L_G$ is an integral graph if and only if $Q(G)$ only has integral eigenvalues.
We call $Q(= Q(G))$ the {\em signless Laplacian matrix} of $G$. 
A {\em $Q$-integral graph} is a graph whose signless Laplacian matrix has only integral eigenvalues. 

In~\cite{Sta07}, Stani\'{c} studied $Q$-integral graphs and found all $172$ connected $Q$-integral graphs with up to $10$ vertices. 
After that, in~\cite{SS08}, Simi\'{c} and Stani\'{c} classified $Q$-integral graphs with maximum edge-degree at most $4$, where the \emph{edge-degree} of an edge of a graph $G$ is the number of edges incident to the edge. 
They also gave some partial results of the classification of $Q$-integral graphs with maximum edge-degree $5$. 
However, the classification of Q-integral graphs with a maximum edge-degree at least $5$ is still open. There are related studies by~\cite{PS19, PS23, PS24}.

In~\cite{PS19}, Park and Sano used the $Q$-spectral radius to classify $Q$-integral graphs, where the \emph{$Q$-spectral radius} of a graph $G$ is the largest eigenvalue of $Q(G)$.
A graph is called \emph{edge-non-regular} if it has at least two edges with different edge-degrees.
In~{\cite{SS08,PS19,OPPS}}, the authors proved that if for each connected $Q$-integral graph with $Q$ spectral radius $6$, then it is one of the known 17 graphs or an edge-non-regular graph with the maximum edge-degree $5$.
Under the assumption that $G$ is a connected $Q$-integral graph with spectral radius $6$ and the maximum edge-degree $5$,
we proved the following theorem:
\begin{thm}\label{thm:fish}
Let $G$ be a connected $Q$-integral graph with $Q$-spectral radius $6$. 
Suppose that $G$ has the maximum edge-degree $5$. 
If $G$ is non-bipartite, then $G$ is the graph with 
the vertex set $V(G) = \{v_1, v_2, v_3, v_4, v_5, v_6 \}$ and 
the edge set $E(G) = \{v_1v_2, v_1v_3, v_2v_3, v_3v_4, v_3v_5, v_4v_5, v_4v_6, v_5v_6 \}$ (see Figure~\ref{fig:fish}). 
\end{thm}

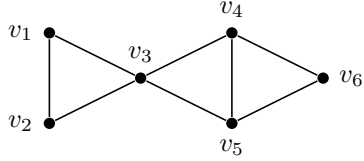
\begin{figure}[!ht]
    \centering
      \begin{tikzpicture}[auto, node distance=2cm, semithick,scale=0.3]
        \tikzstyle{vertex}=[circle,fill,inner sep=1.5pt]
        \node[vertex,label=left:$v_1$] (x0) at (-6, 3)  {};
        \node[vertex,label=above:$v_3$] (x)  at (-2, 1) {};
        \node[vertex,label=below:$v_5$] (y)  at (2, -1) {};
    
        \node[vertex,label=left:$v_2$] (x1) at (-6,-1) {};
    
        \node[vertex,label=above:$v_4$] (y1) at (2, 3) {};
        \node[vertex,label=right:$v_6$] (y0) at (6, 1) {};
        
        \path[-]
        (x) edge (y)
    
        (x) edge (x0)
        (x) edge (x1)
        (x) edge (y1)
        
        (y) edge (y0)
        (y) edge (y1)
        
        (x0) edge (x1)
        (y0) edge (y1);
      \end{tikzpicture}
    \caption{The striped fish graph.}
    \label{fig:fish}
    \end{figure}

As a continuation of the previous work~\cite{OPPS} and~\cite{SS08}, in this paper, we completely classify the connected non-bipartite $Q$-integral graphs with $Q$-spectral radius at most $6$.


\begin{thm}\label{thm:main}
    Let $G$ be a connected non-bipartite $Q$-integral graph with spectral radius at most $6$.
    Then $G$ is isomorphic to one of the graphs in Figure~\ref{fig:main}.
\end{thm}

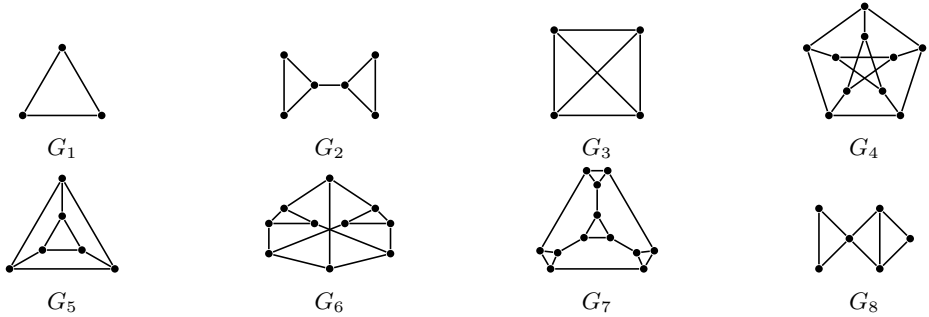
\begin{figure}[!ht]
    \captionsetup[subfigure]{labelformat=empty}
    \centering
    
             
    \begin{subfigure}[b]{0.2\textwidth}
        \centering
      \begin{tikzpicture}[auto, node distance=1cm, semithick, scale=0.2]
        \tikzstyle{vertex}=[circle, fill, inner sep=1pt]
    
        \node[vertex] (c0) at (90 :3) {};
        \node[vertex] (c1) at (210:3) {};
        \node[vertex] (c2) at (330:3) {};
           
        \path[-]
            (c0) edge (c1)
            (c0) edge (c2)
            (c2) edge (c1);
             
      \end{tikzpicture}
      \caption{$G_1$}
    \end{subfigure}
    \begin{subfigure}[b]{0.2\textwidth}
        \centering
      \begin{tikzpicture}[auto, node distance=1cm, semithick, scale=0.2]
        \tikzstyle{vertex}=[circle, fill, inner sep=1pt]
        
        \node[vertex] (x)  at (-1, 0) {};
        \node[vertex] (y)  at (1, 0) {};
        \node[vertex] (x0) at (-3, 2) {};
        \node[vertex] (x2) at (-3,-2) {};
        \node[vertex] (y0) at (3, 2) {};
        \node[vertex] (y2) at (3,-2) {};
        
        \path[-]
        (x) edge (y)
        (x) edge (x0)
        (x) edge (x2)
        (y) edge (y0)
        (y) edge (y2)
        (x0) edge (x2)
        (y0) edge (y2);
             
      \end{tikzpicture}
      \caption{$G_2$}
    \end{subfigure}
    \begin{subfigure}[b]{0.2\textwidth}
        \centering
      \begin{tikzpicture}[auto, node distance=1cm, semithick, scale=0.2]
        \tikzstyle{vertex}=[circle, fill, inner sep=1pt]
    
         \node[vertex] (c0) at (45+90*0:4) {};
         \node[vertex] (c1) at (45+90*1:4) {};
         \node[vertex] (c2) at (45+90*2:4) {};
         \node[vertex] (c3) at (45+90*3:4) {};
         
         \path[-]
           (c0) edge (c1)
           (c1) edge (c2)
           (c2) edge (c3)
           (c3) edge (c0)
           (c0) edge (c2)
           (c1) edge (c3)
           ;
             
      \end{tikzpicture}
      \caption{$G_3$}
    \end{subfigure}
    \begin{subfigure}[b]{0.2\textwidth}
        \centering
      \begin{tikzpicture}[auto, node distance=1cm, semithick, scale=0.2]
        \tikzstyle{vertex}=[circle, fill, inner sep=1pt]
    
        \node[vertex] (n1)  at (90+72*0:2) {};
        \node[vertex] (n2)  at (90+72*1:2) {};
        \node[vertex] (n3)  at (90+72*2:2) {};
        \node[vertex] (n4)  at (90+72*3:2) {};
        \node[vertex] (n5)  at (90+72*4:2) {};

        \node[vertex] (b1)  at (90+72*0:4) {};
        \node[vertex] (b2)  at (90+72*1:4) {};
        \node[vertex] (b3)  at (90+72*2:4) {};
        \node[vertex] (b4)  at (90+72*3:4) {};
        \node[vertex] (b5)  at (90+72*4:4) {};    
    
        \draw (n1) -- (n3);
        \draw (n3) -- (n5);
        \draw (n5) -- (n2);
        \draw (n2) -- (n4);
        \draw (n4) -- (n1);
    
        \draw (b1) -- (b2);
        \draw (b2) -- (b3);
        \draw (b3) -- (b4);
        \draw (b4) -- (b5);
        \draw (b5) -- (b1);
    
        \draw (n1) -- (b1);
        \draw (n2) -- (b2);
        \draw (n3) -- (b3);
        \draw (n4) -- (b4);
        \draw (n5) -- (b5);
             
      \end{tikzpicture}
      \caption{$G_4$}
    \end{subfigure}
    \begin{subfigure}[b]{0.2\textwidth}
        \centering
      \begin{tikzpicture}[auto, node distance=1cm, semithick, scale=0.2]
        \tikzstyle{vertex}=[circle, fill, inner sep=1pt]
    
        \node[vertex] (c0) at (90 :4) {};
        \node[vertex] (c1) at (210:4) {};
        \node[vertex] (c2) at (330:4) {};
        
        \node[vertex] (b0) at (90 :1.5) {};
        \node[vertex] (b1) at (210:1.5) {};
        \node[vertex] (b2) at (330:1.5) {};

        \path[-]
            (c0) edge (c1)
            (c0) edge (c2)
            (c2) edge (c1)
            
            (b0) edge (b1)
            (b0) edge (b2)
            (b2) edge (b1)
            
            (c0) edge (b0)
            (c1) edge (b1)
            (c2) edge (b2)
            ;
             
      \end{tikzpicture}
      \caption{$G_5$}
    \end{subfigure}
      \begin{subfigure}[b]{0.2\textwidth}
        \centering
        \begin{tikzpicture}[auto,node distance=1cm, semithick, scale=0.2]
          \tikzstyle{vertex}=[circle,fill,inner sep=1pt]
              
        \node[vertex] (n1)  at ( 0, 6) {};
        \node[vertex] (n2)  at (-3, 4) {};
        \node[vertex] (n3)  at (-4, 3) {};
        \node[vertex] (n4)  at (-1, 3) {};
        \node[vertex] (n5)  at ( 3, 4) {};
        \node[vertex] (n6)  at ( 1, 3) {};
        \node[vertex] (n7)  at ( 4, 3) {};
        \node[vertex] (n8)  at ( 4, 1) {};
        \node[vertex] (n9)  at (-4, 1) {};
        \node[vertex] (n10) at ( 0, 0) {};
    
    \draw (n1) -- (n2);
    \draw (n1) -- (n5);
    \draw (n2) -- (n3);
    \draw (n3) -- (n4);
    \draw (n4) -- (n2);
    \draw (n5) -- (n6);
    \draw (n6) -- (n7);
    \draw (n7) -- (n5);
    \draw (n3) -- (n9);
    \draw (n7) -- (n8);
    \draw (n4) -- (n8);
    \draw (n6) -- (n9);
    \draw (n1) -- (n10);
    \draw (n9) -- (n10);
    \draw (n8) -- (n10);

        \end{tikzpicture}
        \caption{$G_6$}
      \end{subfigure}
      \begin{subfigure}[b]{0.2\textwidth}
        \centering
        \begin{tikzpicture}[auto,node distance=1cm, semithick, scale=0.2]
          \tikzstyle{vertex}=[circle,fill,inner sep=1pt]
              
        \node[vertex] (c0) at (90 :3) {};
        \node[vertex] (c1) at (210:3) {};
        \node[vertex] (c2) at (330:3) {};
        
        \node[vertex] (b0) at (90 :1) {};
        \node[vertex] (b1) at (210:1) {};
        \node[vertex] (b2) at (330:1) {};

        \node[vertex] (x0) at (90+10:4) {};
        \node[vertex] (x1) at (90-10:4) {};

        \node[vertex] (y0) at (210+10:4) {};
        \node[vertex] (y1) at (210-10:4) {};

        \node[vertex] (z0) at (330+10:4) {};
        \node[vertex] (z1) at (330-10:4) {};

        \path[-]
            (b0) edge (b1)
            (b0) edge (b2)
            (b2) edge (b1)
            
            (c0) edge (b0)
            (c1) edge (b1)
            (c2) edge (b2)

            (c0) edge (x0)
            (c1) edge (y0)
            (c2) edge (z0)
            (c0) edge (x1)
            (c1) edge (y1)
            (c2) edge (z1)

            (x0) edge (x1)
            (y0) edge (y1)
            (z0) edge (z1)

            (x0) edge (y1)
            (y0) edge (z1)
            (z0) edge (x1)
            ;
        \end{tikzpicture}
        \caption{$G_7$}
      \end{subfigure}
      \begin{subfigure}[b]{0.2\textwidth}
        \centering
        \begin{tikzpicture}[auto,node distance=1cm, semithick, scale=0.2]
          \tikzstyle{vertex}=[circle,fill,inner sep=1pt]
              
        \node[vertex] (x0) at (-4, 3)  {};
        \node[vertex] (x)  at (-2, 1) {};
        \node[vertex] (y)  at (0, -1) {};
        \node[vertex] (x1) at (-4,-1) {};
        \node[vertex] (y1) at (0, 3) {};
        \node[vertex] (y0) at (2, 1) {};
        
        \path[-]
        (x) edge (y)
    
        (x) edge (x0)
        (x) edge (x1)
        (x) edge (y1)
        
        (y) edge (y0)
        (y) edge (y1)
        
        (x0) edge (x1)
        (y0) edge (y1);
        \end{tikzpicture}
        \caption{$G_8$}
      \end{subfigure}
  \caption{Connected non-bipartite $Q$-integral graphs with $Q$-spectral radius at most $6$}
    \label{fig:main}
    \end{figure}

\begin{table}[!ht]
  \centering
  \begin{tabular}{c|c|c|ccccc}
    \toprule
  Graph & Number of vertices & Number of edges & \multicolumn{5}{c}{Q-eigenvalues} \\
  \midrule
  $G_{1}$ & 3  & 3  & $4^1$ & $1^2$ & \\
  $G_{2}$ & 6  & 7  & $5^1$ & $4^1$ & $2^1$ & $1^3$ \\
  $G_{3}$ & 4  & 6  & $6^1$ & $2^3$ & \\
  $G_{4}$ & 10 & 15 & $6^1$ & $4^5$ & $1^4$ & \\
  $G_{5}$ & 6  & 9  & $6^1$ & $4^1$ & $3^2$ & $1^2$ \\
  $G_{6}$ & 10 & 15 & $6^1$ & $5^1$ & $4^3$ & $2^2$ & $1^3$ \\
  $G_{7}$ & 12 & 18 & $6^1$ & $5^3$ & $3^2$ & $2^3$ & $1^3$ \\
  $G_{8}$ & 6  & 8  & $6^1$ & $4^1$ & $2^2$ & $1^2$ \\
    \bottomrule
\end{tabular}
  \caption{Information of connected non-bipartite $Q$-integral graphs with $Q$-sepctral radius at most $6$}
\end{table}

This paper is organized as follows. 
In Section~\ref{sec:pre}, we prepare definitions and preliminaries, including properties of $Q$-integral graphs. 
The proofs of Theorem~\ref{thm:fish} and Theorem~\ref{thm:main} will be given in Section~\ref{sec:main}.
First, we present previous results on $Q$-integral graphs with $Q$-spectral radius at most $5$.
The proof methodology is applied iteratively to each induced subgraph of $Q$-integral graphs with a $Q$-spectral radius of $6$. To do this, we explore possible induced subgraphs of non-bipartite $Q$-integral graphs with a $Q$-spectral radius of 6. The proof idea will be explained in detail.
Also, we implemented this iterative proof methodology using a computer to check whether a non-bipartite $Q$-integral graph containing a given induced subgraph exists, and the pseudocode is provided in the Appendix.
It can be used for finding non-bipartite $Q$-integral graphs with $Q$-spectral radius larger than $6$.

\section{Preliminaries} \label{sec:pre}

Let $G$ be a graph with $n$ vertices. We denote the vertex set of $G$ by $V(G)$ and the edge set of $G$ by $E(G)$.
The \emph{adjacency matrix} $A(G) = (a_{xy})$ of $G$ is the $n \times n$ matrix whose rows and columns are indexed by the vertices of $G$ and $a_{xy}=1$ whenever $x$ and $y$ are adjacent and $a_{xy}=0$ otherwise. 
The {\em eigenvalues} of a graph $G$ are those of $A(G)$.
The \emph{degree} $\deg_{G}(x)$ of a vertex $x$ in a graph $G$ is the number of vertices adjacent to $x$.
We denote the complete graph with $n$ vertices, the complete bipartite graph with two parts of sizes $n$ and $m$, and the cycle with $n$ vertices by $K_n$, $K_{n,m}$ and $C_n$, respectively. 
The {\it subdivision} $S(G)$ of $G$ is the graph obtained from $G$ by replacing all the edges with paths of length $2$. The {\it Cartesian product} of graphs $G$ and $H$ is the graph $G\square H$ with vertices $V(G\square H) = V(G)\times V(H)$, and for which $(x,u)(y,v)$ is an edge if $x=y$ and $uv\in E(H)$, or $xy\in E(G)$ and $u=v$.

We recall some results on the signless Laplacian eigenvalues of graphs that we use in this paper. For unexplained terminologies, examples, and
more details, see \cite{PS19}. We will use the Perron-Frobenius Theorem~\cite[Theorem 8.8.1]{god} and the Interlacing Theorem~\cite[Theorem 9.1.1]{god}, which are powerful tools for comparing eigenvalues of a graph and its induced subgraphs.

Let $G$ be a graph and $Q(=Q(G))$ be the signless Laplacian matrix of $G$. 
Let $W$ be a non-empty subset of the vertex set $V(G)$ and $M$ be the principal submatrix of $Q$ indexed by $W$.
Note that for any $y\in W$, the $(y,y)$-entry of $M$ is $\deg_G(y)$. This means that the matrix $M$ may not equal the signless Laplacian matrix of the subgraph of $G$ induced by $W$. Because of this situation, we need the following notations.

A \emph{$Q$-graph} is  a pair $(G,d)$ of a graph $G$ and a function $d: V(G) \to \mathbb{Z}_{\geq 0}$ such that $d(x) \geq \deg_G(x)$ for any $x \in V(G)$. An (induced) \emph{$Q$-subgraph} of a graph $G$ is an (induced) subgraph $H$ of $G$ together with a function $d:V(H) \to \mathbb{Z}_{\geq 0}$ defined by $d(x) = \deg_G(x)$. Note that any $Q$-subgraph of a graph is a $Q$-graph. The \emph{$Q$-matrix} of a $Q$-graph $(G,d)$ is the matrix $Q(G,d)$ defined by 
\[
  (Q(G,d))_{xy} = \left\{
    \begin{array}{ll}
      d(x) & x = y \\
      1 & xy \in E(G) \\
      0 & xy \not\in E(G). \\
    \end{array}
  \right.
\]
The \emph{eigenvalues} of a $Q$-graph $(G,d)$ are 
the eigenvalues of the $Q$-matrix $Q(G,d)$.
Note that for a connected graph $G$, the $Q$-matrix of $(G,d)$ is irreducible since the adjacency matrix of $G$ is irreducible.

The $Q$-matrix of an induced $Q$-subgraph $H$ of a graph $G$ 
is equal to the principal submatrix of the signless Laplacian matrix $Q(G)$ 
of the graph $G$, where rows and columns are restricted to $V(H)$.
In this case, we denote the $Q$-matrix of $H$ by $ Q(G)|_{V(H)}$.
Definitions of $Q$-graph and $Q$-matrix are first introduced by Park and Sano \cite{PS19}.


Proposition~3.11 of \cite{PS19} is the original version of Proposition~\ref{prop:ev}, which includes the bipartite case. 
However, in this paper, we investigate non-bipartite graphs with larger minimum eigenvalues compared to bipartite graphs.
This proposition is frequently used in this paper.

\begin{prop}[{\cite[Proposition~2.3]{OPPS}}, {\cite[Proposition 3.11]{PS19}} and {\cite[Proposition~3.7]{PS23}}]\label{prop:ev}
  Let $G$ be a connected non-bipartite $Q$-integral graph with $Q$-spectral radius $\rho$.
  If $H$ is a connected induced $Q$-subgraph of $G$,
  then the following hold: 
  \begin{itemize}
  \item[{\rm (i)}]
    The largest eigenvalue of the $Q$-matrix of $H$ is at most $\rho$, 
    and 
    is equal to $\rho$ if and only if $G = H$; 
  \item[{\rm (ii)}]
    The second largest eigenvalue of the $Q$-matrix of $H$ is at most $\rho-1$; 
  \item[{\rm (iii)}]
    The smallest eigenvalue of the $Q$-matrix of $H$ is at least $1$.
  \end{itemize}
\end{prop}

If the $Q$-spectral radius of graphs is bounded, then both the maximum vertex degree and the maximum edge-degree of the graph are also bounded. 
The result is as follows.


\begin{prop}[{\cite[Proposition 2.7-8.]{PS19}}]\label{prop:verdeg}
Let $G$ be a connected non-bipartite $Q$-integral graph with $Q$-spectral radius $\rho$. 
Then the maximum vertex degree of $G$ is at most $\rho-2$, 
and the maximum edge-degree of $G$ is at most $2\rho - 6$.

\end{prop}



For an arbitrary subgraph $H$ of $G$, let $G[V(H)]$ denote the subgraph of $G$ induced by $V(H)$. 
We set $H'=G[V(H)\cup S]$ as a $Q$-subgraph for some subset $S\subseteq V(G)\setminus V(H)$. Then the $Q$-matrix of $H'$ is 
\[
  Q(G)|_{V(H')} = 
  \begin{pmatrix}
    Q(G)|_{V(H)} & A_{H'}^T \\
    A_{H'} & Q_{H'}
  \end{pmatrix},
\]
where the block matrices $A_{H'}$, $A_{H'}^T$, and $Q_{H'}$ are as follows:
\begin{itemize}
  \item $A_{H'}$ is an $|S|\times |V(H)|$ matrix whose rows and columns are indexed by the elements of $S$ and the vertices of $H$ respectively, and the $(u,v)$-entry of $A_{H'}$ is $1$ whenever $u$ and $v$ are adjacent and $0$ otherwise,
  \item $A_{H'}^T$ is the transpose of $A_{H'}$, 
  \item $Q_{H'}$ is the $Q$-matrix of the $Q$-subgraph of $G$ induced by $S$.
\end{itemize}

We note that $a_{uv}$ and $d(w)$ are abbreviated as $a_{\cdot\cdot}$ and $d(\cdot)$, respectively, for appropriate vertices $u, v$ and $w$.

\section{Proof of Theorem~\ref{thm:fish}}\label{sec:main}

If a $Q$-integral graph has a $Q$-spectral radius at most $5$, then by Proposition~\ref{prop:verdeg}, the maximum edge-degree is at most $4$. 
All $Q$-integral graphs with a maximum edge-degree at most $4$ were classified in~\cite{SS08}. 
The complete list of all connected non-bipartite $Q$-integral graphs with a $Q$-spectral radius at most $5$ is provided in Table 2.

\begin{table}[!ht] 
\centering
\begin{tabular}{c|c|c|c}
     \toprule
     $Q$-spectral radius & $Q$-integral graphs  & graph in Figure~\ref{fig:main} & maximum edge-degree\\
     \midrule
     4 & $C_3$  & $G_1$ & 2 \\
     5 & $H_2$  & $G_2$ & 4 \\
     \bottomrule
\end{tabular}
\caption{Connected non-bipartite $Q$-integral graphs with $Q$-spectral radius at most $5$}
\label{tab:list}
\end{table}

For $Q$-integral graphs with $Q$-spectral radius $6$, there is a result as follows:

\begin{thm}[{\rm \cite{SS08,PS19,OPPS}}]\label{thm:PS19}
  Let $G$ be a connected $Q$-integral graph with $Q$-spectral radius $6$. 
  Then, one of the following holds:
  \begin{itemize}
  \item[\rm (i)] 
  $G$ is one of the $17$ graphs: 
  $K_{1,5}$, $K_{2,4}$, $S(K_5)$, $K_{1,3} \square K_2$, and the $13$ cubic integral graphs;
  \item[\rm (ii)] 
  $G$ is edge-non-regular with the maximum edge-degree $5$.
  \end{itemize}
  \end{thm}

We consider the case (ii) of Theorem~\ref{thm:PS19} under the assumption that $G$ is non-bipartite.


Let $G$ be a connected non-bipartite $Q$-integral graph with $Q$-spectral radius $6$ and maximum edge-degree $5$.
Let $xy$ be an edge of edge-degree $5$ in $G$. 
By Proposition~\ref{prop:verdeg},
the maximum vertex degree of $G$ is at most $4$.
Without loss of generality, 
we may assume that $\text{deg}_G(x)=4$ and $\text{deg}_G(y)=3$.
Therefore, $x$ and $y$ have at most two common neighbors.
We use these two vertices in the whole proof of Theorem~\ref{thm:fish}.

To prove Theorem~\ref{thm:fish}, we use the following method.
At first, consider all possible $Q$-subgraphs $H$ of $G$ induced by $x$, $y$, and their neighbors, and check if any $H$ satisfies Proposition~\ref{prop:ev}.
If no such possible $Q$-subgraphs $H$ satisfies Proposition~\ref{prop:ev}, then we conclude that there is no $Q$-integral graph with $Q$-spectral radius $6$ and maximum edge-degree $5$ containing $H$ as an induced $Q$-subgraph.
If there exists a $Q$-subgraph $H$ satisfying Proposition~\ref{prop:ev}, we check whether $H$ is a $Q$-integral graph with $Q$-spectral radius $6$. 
In this case, if the $Q$-subgraph $H$ is not a $Q$-integral graph with $Q$-spectral radius $6$, then we create all possible connected induced $Q$-subgraphs $\tilde{H}$ of $G$ by adding a vertex to $H$. 
Repeat the above process starting from $\tilde{H}$.


\subsection{The case where $x$ and $y$ have two common neighbors} \label{sec:2common}

In this subsection, 
we will show that $x$ and $y$ have no two common neighbors.

\begin{lem}\label{2common}
  Let $G$ be a connected non-bipartite $Q$-integral graph with $Q$-spectral radius $6$ and maximum edge-degree $5$. 
  Then, $x$ and $y$ have no two common neighbors.
\end{lem}
\begin{proof}
  Assume that $x$ and $y$ have two common neighbors, $y_0$ and $y_1$.
  Since the vertex $x$ has degree $4$,
  there exists a vertex $x_0$ that is adjacent to $x$
  but not adjacent to $y$ (see Figure~\ref{fig:2common}).
  
  \begin{figure}[ht]
    \centering
    \begin{tikzpicture}[auto,node distance=2cm,semithick,scale=0.3]
      \tikzstyle{vertex}=[circle,fill,inner sep=1.5pt]
      \node[vertex] (x0) at (-6, 0)  {};
      \node[vertex] (x)  at (-2, 0) {};
      \node[vertex] (y)  at (2, 0) {};
  
      \node[vertex] (x1) at (0,3) {};
      \node[vertex] (x2) at (0,-3) {};
  
      \node[vertex] (y1) at (0,3) {};
      \node[vertex] (y0) at (0, -3) {};
      
      \path[-]
      (x) edge (y)
  
      (x) edge (x0)
      (x) edge (x1)
      (x) edge (x2)
      
      (y) edge (y0)
      (y) edge (y1);

  
      \draw
      (x) node[below, yshift=-3, scale=0.8] {$x$}
      (y) node[right, scale=0.8] {$y$}
  
      (x0) node[left, scale=0.8] {$x_0$}
    
      (y0) node[below, scale=0.8] {$y_0$}
      (y1) node[above, scale=0.8] {$y_1$};
    \end{tikzpicture}
    \caption{$x$ and $y$ have two common neighbors $y_0$ and $y_1$.}
    \label{fig:2common}
  \end{figure}
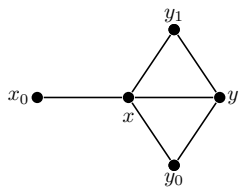 

  Let $H$ be the $Q$-subgraph induced by $\{x$, $y$, $x_0$, $y_0$, $y_1\}$.
  Note that all neighbors of $x$ cannot have degree $4$, if then the edge-degree of the edge exceed $5$.
  Then the $Q$-matrix of $H$ has the following form:
  \[
  Q(G)|_{V(H)} =
  \begin{pmatrix}
    4 & 1 & 1          & 1          & 1\\
    1 & 3 & 0          & 1          & 1\\
    1 & 0 & d(x_0)     & a_{x_0y_0} & a_{x_0y_1} \\
    1 & 1 & a_{y_0x_0} & d(y_0)     & a_{y_0y_1} \\ 
    1 & 1 & a_{y_1x_0} & a_{y_1y_0} & d(y_1)
  \end{pmatrix},
  \]
  where $d(x_0)\in\{1,2,3\}$, $d(\cdot)\in\{2,3\}$ and $a_{\cdot\cdot}\in\{0,1\}$.
  If $d(x_0)=1$ or $a_{uv}=1$ for some two vertices $u$ and $v$ in $\{x_0, y_0, y_1\}$, then the induced $Q$-matrix of $H$ dose not satisfy the Proposition~\ref{prop:ev}. 
  So, we have $d(x_0) \in \{2,3\}$ and $a_{\cdot\cdot}=1$.
  Also, if $d(y_0)=d(y_1)$, then it dose not satisfy the Proposition~\ref{prop:ev}. 
  Thus, we have $(d(y_0), d(y_1)) = (2,3)$ or $(3,2)$.
  By the symmetry, we may assume that $d(y_0)=2$ and $d(y_1)=3$.
  Then there is a vertex $z\in V(G)\setminus V(H)$ such that $z$ is adjacent to $y_1$. 
  We set $S=\{z\}$ and $H'=G[V(H)\cup S]$.
  Then the $Q$-matrix of $H'$ is 
  \[
  Q(G)|_{V(H')} =
  \begin{pmatrix}
    Q(G)|_{V(H)} & A_{H'}^T \\
    A_{H'}  & Q_{H'}
  \end{pmatrix}
  =
  \left(
  \begin{array}{ccccc|c}
    4 & 1 & 1        & 1 & 1 & 0\\
    1 & 3 & 0        & 1 & 1 & 0\\
    1 & 0 & d(x_0)   & 0 & 0 & a_{x_0z}\\
    1 & 1 & 0        & 2 & 0 & 0\\ 
    1 & 1 & 0        & 0 & 3 & 1\\ \hline
    0 & 0 & a_{zx_0} & 0 & 1 & d(z)
  \end{array}
  \right),
  \]
  where $d(x_0)\in\{2,3\}$, $d(z)\in\{1,2,3,4\}$ and $a_{x_0z}\in\{0,1\}$.
  Every case except $d(x_0)=3$, $d(z)=4$, $a_{x_0z}=0$ has the smallest eigenvalue less than $1$.
  But the largest eigenvalue of $Q(H')$ in the case $d(x_0)=3$, $d(z)=4$, $a_{x_0z}=0$ is greater than $6$.
  By Interlacing Theorem, the induced subgraph $H$ of $G$ cannot exist, i.e., $x$ and $y$ have no two common neighbors. This completes the proof.
\end{proof}

\subsection{The case where $x$ and $y$ have one common neighbor}\label{sec:1common}
We assume that $x$ and $y$ have exactly one common neighbor.
In that case, $x$, $y$, and their common neighbor form a triangle.
We define the graph $T_{3,2}$ as follows (see Figure~\ref{fig:T32}):
\begin{eqnarray*} 
  V(T_{3,2}) &=& \{x_0, x_1, x, y, y_0, y_1\}, \\
  E(T_{3,2}) &=& \{x_0x, x_1x, y_1x, xy, yy_0, yy_1\}. 
\end{eqnarray*}


\begin{figure}[ht]
  \centering
  \begin{tikzpicture}[auto,node distance=2cm,semithick,scale=0.3]
    \tikzstyle{vertex}=[circle,fill,inner sep=1.5pt]
    \node[vertex] (x0) at (-6, 0)  {};
    \node[vertex] (x)  at (-2, 0) {};
    \node[vertex] (y)  at (2, 0) {};

    \node[vertex] (x2) at (0,3) {};
    \node[vertex] (x1) at (-2,-4) {};

    \node[vertex] (y1) at (0, 3) {};
    \node[vertex] (y0) at (6, 0) {};
    
    \path[-]
    (x) edge (y)

    (x) edge (x0)
    (x) edge (x1)
    (x) edge (x2)
    
    (y) edge (y0)
    (y) edge (y1);


    \draw
    (x) node[right, yshift=-6, scale=0.8] {$x$}
    (y) node[left, yshift=-6, scale=0.8] {$y$}

    (x0) node[left, scale=0.8] {$x_0$}
    (x1) node[left, scale=0.8] {$x_1$}

    (y0) node[right, scale=0.8] {$y_0$}
    (y1) node[above, scale=0.8] {$y_1$};
  \end{tikzpicture}
  \caption{$T_{3,2}$}
  \label{fig:T32}
\end{figure}

\begin{lem} \label{lem:t32_2edge}
  Let $G$ be a connected non-bipartite $Q$-integral graph 
  with $Q$-spectral radius $6$ and maximum edge-degree $5$.
  Suppose that $G$ contains $T_{3,2}$ as a subgraph. 
  Then $|E(G[V(T_{3,2})])\setminus E(T_{3,2})|\leq2$. 
  Moreover, if $|E(G[V(T_{3,2})])\setminus E(T_{3,2})|=2$, then $G$ is isomorphic to the striped fish graph in Figure~\ref{fig:fish}.
\end{lem}
\begin{proof}
  The $Q$-matrix of $T_{3,2}$ has the following form:
  \[
  Q(G)|_{V(T_{3,2})}=
  \begin{pmatrix}
    4 & 1 & 1          & 1          & 0          & 1         \\
    1 & 3 & 0          & 0          & 1          & 1         \\
    1 & 0 & d(x_0)     & a_{x_0x_1} & a_{x_0y_0} & a_{x_0y_1}\\
    1 & 0 & a_{x_1x_0} & d(x_1)     & a_{x_1y_0} & a_{x_1y_1}\\
    0 & 1 & a_{y_0x_0} & a_{y_0x_1} & d(y_0)     & a_{y_0y_1}\\
    1 & 1 & a_{y_1x_0} & a_{y_1x_1} & a_{y_1y_0} & d(y_1)
  \end{pmatrix},
  \]
  where $d(x_0),d(x_1)\in\{1,2,3\}$, $d(y_0)\in\{1,2,3,4\}$, $d(y_1)\in\{2,3\}$, and $a_{\cdot\cdot}\in\{0,1\}$.
  Applying Proposition~\ref{prop:ev} to $G[V(T_{3,2})]$, 
  we have $d(x_0),d(x_1),d(y_1)\in\{2,3\}$ and $d(y_0)\in\{2,3,4\}$.
  Also, we have $(1,0,0,0,0,1)$, $(1,0,0,0,0,0)$, $(0,1,0,0,0,0)$, $(0,0,0,1,0,0)$ and $(0,0,0,0,0,0)$ for $(a_{x_0x_1}$,$a_{x_0y_0}$,$a_{x_0y_1}$,$a_{x_1y_0}$,$a_{x_1y_1}$, $a_{y_0y_1})$. 
  This shows that $|E(G[V(T_{3,2})])\setminus E(T_{3,2})|\leq2$.

  Now, we consider the case $|E(G[V(T_{3,2})])\setminus E(T_{3,2})|=2$, i.e., $(a_{x_0x_1}$, $a_{x_0y_0}$, $a_{x_0y_1}$, $a_{x_1y_0}$, $a_{x_1y_1}$, $a_{y_0y_1})=(1,0,0,0,0,1)$.
  In this case, we have $d(x_0)=d(x_1)=d(y_0)=2$ and $d(y_1)=3$ and the spectrum is $\{6^1, 4^1, 2^2, 1^2\}$. 
  Hence, the graph $G$ is isomorphic to the striped fish graph in Figure~\ref{fig:fish}. 
  This completes the proof.
\end{proof}

\begin{lem} \label{lem:t32x1y0}
  Let $G$ be a connected non-bipartite $Q$-integral graph with $Q$-spectral radius $6$
  and maximum edge-degree $5$.
  Suppose that $G$ contains $T_{3,2}$ as a subgraph. 
  Then we have $E(G[V(T_{3,2})])\setminus E(T_{3,2}) \neq \{x_1y_0\}$.
\end{lem}
\begin{proof}
  Assume that $E(G[V(T_{3,2})])\setminus E(T_{3,2}) = \{x_1y_0\}$. 
  We denote by $T_{3,2}'$ the graph obtained from $T_{3,2}$ by adding the edge $x_1y_0$ (see Figure~\ref{fig:t32x1y0_1}).
  
  \begin{figure}[!ht]
      \centering
      \begin{tikzpicture}[auto,node distance=2cm,semithick,scale=0.3]
        \tikzstyle{vertex}=[circle,fill,inner sep=1.5pt]
        \node[vertex] (x0) at (-6, 0)  {};
        \node[vertex] (x)  at (-2, 0) {};
        \node[vertex] (y)  at (2, 0) {};
    
        \node[vertex] (x2) at (0,3) {};
        \node[vertex] (x1) at (-2,-4) {};
    
        \node[vertex] (y1) at (0, 3) {};
        \node[vertex] (y0) at (6, 0) {};
        
        \path[-]
        (x) edge (y)
    
        (x) edge (x0)
        (x) edge (x1)
        (x) edge (x2)
        
        (y) edge (y0)
        (y) edge (y1)
        
        (x1) edge (y0);
    
        \draw
        (x) node[right, yshift=-6, scale=0.8] {$x$}
        (y) node[left, yshift=-6, scale=0.8] {$y$}
    
        (x0) node[left, scale=0.8] {$x_0$}
        (x1) node[left, scale=0.8] {$x_1$}
    
        (y0) node[right, scale=0.8] {$y_0$}
        (y1) node[above, scale=0.8] {$y_1$};
      \end{tikzpicture}
      \caption{$T_{3,2}'$}
      \label{fig:t32x1y0_1}
  \end{figure}
  If $d(x_0)=1$, then the largest eigenvalue is greater than $6$ or the smallest eigenvalue is less than $1$. 
  By Proposition~\ref{prop:ev}, we have $d(x_0)\in\{2,3\}$.
  Since $d(x_0)\geq2$, there exists $z_1$ in $V(G)\setminus V(T_{3,2}')$ such that 
  $z_1$ is adjacent to $x_0$. We set $S_1=\{z_1\}$ and $H'=G[V(T_{3,2}')\cup S_1]$.
  Then the $Q$-matrix of $H'$ has the following form:
  \[
  Q(G)|_{V(H')}=
  \begin{pmatrix}
    Q(G)|_{V(T_{3,2}')} & A_{H'}^T \\
    A_{H'}  & Q_{H'}
  \end{pmatrix},
  \]
  where
  \[
  Q(G)|_{V(T_{3,2}')}=
  \begin{pmatrix}
    4 & 1 & 1       & 1      & 0      & 1 \\
    1 & 3 & 0       & 0      & 1      & 1 \\
    1 & 0 & d(x_0)  & 0      & 0      & 0 \\
    1 & 0 & 0       & d(x_1) & 1      & 0 \\
    0 & 1 & 0       & 1      & d(y_0) & 0 \\
    1 & 1 & 0       & 0      & 0      & d(y_1)
  \end{pmatrix},
  \]
  \[
    A_{H'}  = 
    \begin{pmatrix}
      0 & 0 & 1 & a_{z_1x_1} & a_{z_1y_0} & a_{z_1y_1} \\ 
    \end{pmatrix}\textrm{~and~}
    Q_{H'} = 
    \begin{pmatrix}
      d(z_1) 
    \end{pmatrix}.
  \] 

  For any $d(\cdot)$ and $a_{\cdot\cdot}$, 
  applying Proposition~\ref{prop:ev} to $H'$, 
  we have $(a_{z_1x_1}$,$a_{z_1y_0}$,$a_{z_1y_1})=(0,1,0)$ or $(0,0,0)$ (see Figure~\ref{fig:t32x1y0_level2}).

  \begin{figure}[ht]
  \centering
  \begin{subfigure}[b]{0.45\textwidth}
    \centering
      \begin{tikzpicture}[auto,node distance=2cm,semithick,scale=0.2]
        \tikzstyle{vertex}=[circle,fill,inner sep=1.5pt]
        \node[vertex] (x0) at (-6, 0)  {};
        \node[vertex] (x)  at (-2, 0) {};
        \node[vertex] (y)  at (2, 0) {};
    
        \node[vertex] (x2) at (0,3) {};
        \node[vertex] (x1) at (-2,-4) {};
    
        \node[vertex] (y1) at (0, 3) {};
        \node[vertex] (y0) at (6, 0) {};
        
        \node[vertex] (z1) at (-6, -4) {};

        \path[-]
        (x) edge (y)
    
        (x) edge (x0)
        (x) edge (x1)
        (x) edge (x2)
        
        (y) edge (y0)
        (y) edge (y1)
        
        (x1) edge (y0)
        (x0) edge (z1)
        (z1) edge (y0);
    
        \draw
        (x) node[right, yshift=-6, scale=0.8] {$x$}
        (y) node[left, yshift=-6, scale=0.8] {$y$}
    
        (x0) node[left, scale=0.8] {$x_0$}
        (x1) node[left, scale=0.8] {$x_1$}
    
        (y0) node[right, scale=0.8] {$y_0$}
        (y1) node[above, scale=0.8] {$y_1$}

        (z1) node[left, scale=0.8] {$z_1$};
      \end{tikzpicture}
      \caption{$(a_{z_1x_1}$, $a_{z_1y_0}$, $a_{z_1y_1})=(0,1,0)$}
      \label{fig:t32x1y0_3}
  \end{subfigure}
  \hskip1cm
  \begin{subfigure}[b]{0.45\textwidth}
    \centering
      \begin{tikzpicture}[auto,node distance=2cm,semithick,scale=0.2]
        \tikzstyle{vertex}=[circle,fill,inner sep=1.5pt]
        \node[vertex] (x0) at (-6, 0)  {};
        \node[vertex] (x)  at (-2, 0) {};
        \node[vertex] (y)  at (2, 0) {};
    
        \node[vertex] (x2) at (0,3) {};
        \node[vertex] (x1) at (-2,-4) {};
    
        \node[vertex] (y1) at (0, 3) {};
        \node[vertex] (y0) at (6, 0) {};
        
        \node[vertex] (z1) at (-6, -4) {};

        \path[-]
        (x) edge (y)
    
        (x) edge (x0)
        (x) edge (x1)
        (x) edge (x2)
        
        (y) edge (y0)
        (y) edge (y1)
        
        (x1) edge (y0)
        (x0) edge (z1);
    
        \draw
        (x) node[right, yshift=-6, scale=0.8] {$x$}
        (y) node[left, yshift=-6, scale=0.8] {$y$}
    
        (x0) node[left, scale=0.8] {$x_0$}
        (x1) node[left, scale=0.8] {$x_1$}
    
        (y0) node[right, scale=0.8] {$y_0$}
        (y1) node[above, scale=0.8] {$y_1$}

        (z1) node[left, scale=0.8] {$z_1$};
      \end{tikzpicture}
      \caption{$(a_{z_1x_1}$, $a_{z_1y_0}$, $a_{z_1y_1})=(0,0,0)$}
      \label{fig:t32x1y0_2}
    \end{subfigure}
  \caption{$H'$ : $T_{3,2}'$ with $z_1$}
  \label{fig:t32x1y0_level2}
  \end{figure}
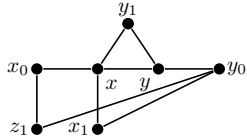
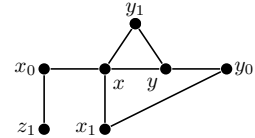

  If $(a_{z_1x_1}$,$a_{z_1y_0}$,$a_{z_1y_1})=(0,1,0)$,
  then $(d(x_0)$, $d(x_1)$, $d(y_1)$, $d(y_0)$, $d(z_1))=(2,3,3,2,4)$
  by Proposition~\ref{prop:ev}.
  Since $d(x_1)=3$, there exists $z_2$ in $V(G)\setminus V(H')$ such that $z_2$ is adjacent to $x_1$. We set $S_2=S_1\cup\{z_2\}$ and $H''=G[V(T_{3,2}')\cup S_2]$. Then, the $Q$-matrix of $H''$ does not satisfy Proposition~\ref{prop:ev}.
  
  If $(a_{z_1x_1}$,$a_{z_1y_0}$,$a_{z_1y_1})=(0,0,0)$,
  then $d(x_0)=3$ by Proposition~\ref{prop:ev}.
  Since $d(x_0)=3$, there exists $z_2$ in $V(G)\setminus V(H')$ such that $z_2$ is adjacent to $x_0$. We set $S_2=S_1\cup\{z_2\}$ and $H''=G[V(T_{3,2}')\cup S_2]$. Then the $Q$-matrix of $H''$ has the following form:
  \[
  Q(G)|_{V(H'')}=
  \begin{pmatrix}
    Q(G)|_{V(T_{3,2}')} & A_{H''}^T \\
    A_{H''}  & Q_{H''}
  \end{pmatrix},
  \]
  where
  \[
    (A_{H''}|Q_{H''})  = 
    \left(
    \begin{array}{cccccc|cc}
      0 & 0 & 1 & 0 & 0 & 0 & d(z_1)     & a_{z_1z_2} \\
      0 & 0 & 1 & a_{z_2x_1} & a_{z_2y_0} & a_{z_2y_1} & a_{z_2z_1} & d(z_2)
    \end{array}
    \right).
  \]
  For any $d(\cdot)$ and $a_{\cdot\cdot}$, 
  applying Proposition~\ref{prop:ev} to $H''$, we have 
  $(a_{z_2x_1}$,$a_{z_2y_0}$,$a_{z_2y_1})=(0,0,0)$ and
  $a_{z_2z_1}\in\{0,1\}$ (see Figure~\ref{fig:t32x1y0_level3}).
  
  \begin{figure}[ht]
  \centering
  \begin{subfigure}[b]{0.4\textwidth}
    \centering
      \begin{tikzpicture}[auto,node distance=2cm,semithick,scale=0.2]
        \tikzstyle{vertex}=[circle,fill,inner sep=1.5pt]
        \node[vertex] (x0) at (-6, 0)  {};
        \node[vertex] (x)  at (-2, 0) {};
        \node[vertex] (y)  at (2, 0) {};
    
        \node[vertex] (x2) at (0,3) {};
        \node[vertex] (x1) at (-2,-4) {};
    
        \node[vertex] (y1) at (0, 3) {};
        \node[vertex] (y0) at (6, 0) {};
        
        \node[vertex] (z1) at (-6, -4) {};
        \node[vertex] (z2) at (-10, 0) {};

        \path[-]
        (x) edge (y)
    
        (x) edge (x0)
        (x) edge (x1)
        (x) edge (x2)
        
        (y) edge (y0)
        (y) edge (y1)
        
        (x1) edge (y0)
        (x0) edge (z1)

        (z2) edge (x0)
        (z2) edge (z1);
    
        \draw
        (x) node[right, yshift=-6, scale=0.8] {$x$}
        (y) node[left, yshift=-6, scale=0.8] {$y$}
    
        (x0) node[above, scale=0.8] {$x_0$}
        (x1) node[left, scale=0.8] {$x_1$}
    
        (y0) node[right, scale=0.8] {$y_0$}
        (y1) node[above, scale=0.8] {$y_1$}

        (z1) node[left, scale=0.8] {$z_1$}
        (z2) node[left, scale=0.8] {$z_2$};
      \end{tikzpicture}
      \caption{$(a_{z_2x_1}$, $a_{z_2y_0}$, $a_{z_2y_1}$, $a_{z_2z_1})=$ $(0,0,0,1)$}
      \label{fig:t32x1y0_5}
  \end{subfigure}
  \hskip1.5cm
  \begin{subfigure}[b]{0.4\textwidth}
    \centering
      \begin{tikzpicture}[auto,node distance=2cm,semithick,scale=0.2]
        \tikzstyle{vertex}=[circle,fill,inner sep=1.5pt]
        \node[vertex] (x0) at (-6, 0)  {};
        \node[vertex] (x)  at (-2, 0) {};
        \node[vertex] (y)  at (2, 0) {};
    
        \node[vertex] (x2) at (0,3) {};
        \node[vertex] (x1) at (-2,-4) {};
    
        \node[vertex] (y1) at (0, 3) {};
        \node[vertex] (y0) at (6, 0) {};
        
        \node[vertex] (z1) at (-6, -4) {};
        \node[vertex] (z2) at (-10, 0) {};

        \path[-]
        (x) edge (y)
    
        (x) edge (x0)
        (x) edge (x1)
        (x) edge (x2)
        
        (y) edge (y0)
        (y) edge (y1)
        
        (x1) edge (y0)
        (x0) edge (z1)

        (z2) edge (x0);
    
        \draw
        (x) node[right, yshift=-6, scale=0.8] {$x$}
        (y) node[left, yshift=-6, scale=0.8] {$y$}
    
        (x0) node[above, scale=0.8] {$x_0$}
        (x1) node[left, scale=0.8] {$x_1$}
    
        (y0) node[right, scale=0.8] {$y_0$}
        (y1) node[above, scale=0.8] {$y_1$}

        (z1) node[left, scale=0.8] {$z_1$}
        (z2) node[left, scale=0.8] {$z_2$};
      \end{tikzpicture}
      \caption{$(a_{z_2x_1}$, $a_{z_2y_0}$, $a_{z_2y_1}$, $a_{z_2z_1})=(0,0,0,0)$}
      \label{fig:t32x1y0_4}
    \end{subfigure}
  \caption{$H''$ : $H'$ with $z_2$ when $(a_{z_1x_1}$,$a_{z_1y_0}$,$a_{z_1y_1})=(0,0,0)$.}
  \label{fig:t32x1y0_level3}
  \end{figure}
  
  For any $a_{z_2z_1}\in\{0,1\}$,
  applying Proposition~\ref{prop:ev} to $H''$, we have $d(x_1)=3$.
  Since $d(x_1)=3$, there exists $z_3$ in $V(G)\setminus V(H'')$ such that $z_3$ is adjacent to $x_1$. We set $S_3=S_2\cup\{z_3\}$ and $H'''=G[V(T_{3,2}')\cup S_3]$. Then, the $Q$-matrix of $H'''$ does not satisfy Proposition~\ref{prop:ev}.
\end{proof}

Note that our idea of proofs is simple 
because we start with a specific subgraph and extend the subgraph by adding vertices one by one.
At a certain moment, the extended graph becomes the whole graph satisfying all the conditions of Proposition~\ref{prop:ev}, or the extended graph does not satisfy at least one of the conditions of Proposition~\ref{prop:ev}. 
We implemented this iterative proof methodology using a computer to check whether a non-bipartite $Q$-integral graph containing a given induced subgraph exists, and the pseudocode is provided in the Appendix.
It can be used for general $Q$-spectral radius. 
Although the performance of the implementation will depend on the computational resources available and the programming language used, the classification method for $Q$-integral graphs using it has the potential to classify $Q$-integral graphs with $Q$-spectral radius larger than $6$.
It follows that we either find a non-bipartite $Q$-integral graph with a $Q$-spectral radius of 6 and a maximum edge-degree of 5, or we conclude that such a subgraph cannot exist as an induced subgraph in a non-bipartite $Q$-integral graph with a $Q$-spectral radius of 6 and a maximum edge-degree of 5. Although the idea is simple, we must check many possible cases. This makes the proofs lengthy if we explain all the details step by step. Therefore, in the rest of this paper, we omit the repeated proofs.

\begin{lem}\label{lem:t32not1}
  Let $G$ be a connected non-bipartite $Q$-integral graph with $Q$-spectral radius $6$ and maximum edge-degree $5$.
  Suppose that $G$ contains $T_{3,2}$ as a subgraph. 
  Then $|E(G[V(T_{3,2})])\setminus E(T_{3,2})|\neq 1$. 
\end{lem}
\begin{proof}
  Suppose that $|E(G[V(T_{3,2})])\setminus E(T_{3,2})|=1$.
  By Proposition~\ref{prop:ev} and the symmetry, we have the following two cases:
  \begin{itemize}
  \item[(a)] $E(G[V(T_{3,2})])\setminus E(T_{3,2})=\{x_0x_1\}$;
  \item[(b)] $E(G[V(T_{3,2})])\setminus E(T_{3,2})=\{x_0y_0\}$.
  \end{itemize}

   By applying the method to cases (a) and (b) in the same way as in Lemma~\ref{lem:t32x1y0}, we can conclusively show that no $Q$-integral graph exists.
\end{proof}

\begin{lem}\label{lem:t32not0}
  Let $G$ be a connected non-bipartite $Q$-integral graph with $Q$-spectral radius $6$ and maximum edge-degree $5$.
  Suppose that $G$ contains $T_{3,2}$ as a subgraph. 
  Then $E(G[V(T_{3,2})])\setminus E(T_{3,2}) \not= \emptyset$. 
\end{lem}
\begin{proof}
  Assume that $E(G[V(T_{3,2})])\setminus E(T_{3,2}) = \emptyset$. 
  By Proposition~\ref{prop:ev}, $d(x_0) \geq 2$. 
  Therefore, there is a vertex $z_1 \in V(G) \setminus V(T_{3,2})$ such that $z_1$ is adjacent to $x_0$.
  By Proposition~\ref{prop:ev}, we have only four cases: 
  \begin{itemize}
    \item[(a)] $z_1$ is adjacent to both $x_0$ and $y_1$;
    \item[(b)] $z_1$ is adjacent to both $x_0$ and $x_1$;
    \item[(c)] $z_1$ is adjacent to both $x_0$ and $y_0$;
    \item[(d)] $z_1$ is adjacent to only $x_0$.
  \end{itemize}
      By applying the method to cases (a) - (d) in the same way as in Lemma~\ref{lem:t32x1y0}, we can conclusively show that no $Q$-integral graph exists.
\end{proof}

\subsection{The case where $x$ and $y$ have no common neighbor} \label{sec:0common}
We assume that $x$ and $y$ have no common neighbors. Then we define the graph $S_{3,2}$ as follows (see Figure~\ref{fig:S32}):
\begin{eqnarray*} 
  V(S_{3,2}) &=& \{x_0, x_1, x_2, x, y, y_0, y_1\}, \\
  E(S_{3,2}) &=& \{x_0x, x_1x, x_2x, xy, yy_0, yy_1\}. 
\end{eqnarray*}

\begin{figure}[!ht]
  \centering
  \begin{tikzpicture}[auto,node distance=2cm,semithick,scale=0.3]
    \tikzstyle{vertex}=[circle,fill,inner sep=1.5pt]
    \node[vertex] (x0) at (-6, 0)  {};
    \node[vertex] (x)  at (-2, 0) {};
    \node[vertex] (y)  at (2, 0) {};

    \node[vertex] (x2) at (-2, 4) {};
    \node[vertex] (x1) at (-2,-4) {};

    \node[vertex] (y1) at (5, 3) {};
    \node[vertex] (y0) at (5, -3) {};
    
    \path[-]
    (x) edge (y)

    (x) edge (x0)
    (x) edge (x1)
    (x) edge (x2)
    
    (y) edge (y0)
    (y) edge (y1);


    \draw
    (x) node[right, yshift=-6, scale=0.8] {$x$}
    (y) node[left, yshift=-6, scale=0.8] {$y$}

    (x0) node[left, scale=0.8] {$x_0$}
    (x1) node[left, scale=0.8] {$x_1$}
    (x2) node[left, scale=0.8] {$x_2$}

    (y0) node[right, scale=0.8] {$y_0$}
    (y1) node[right, scale=0.8] {$y_1$};
  \end{tikzpicture}
  \caption{$S_{3,2}$}
  \label{fig:S32}
\end{figure}

\newpage

\begin{lem}\label{lem:s32disjoint}
  Let $G$ be a connected non-bipartite $Q$-integral graph with $Q$-spectral radius $6$ and maximum edge-degree $5$.
  Suppose that $G$ contains $S_{3,2}$ as a subgraph. 
  Then every two distinct edges in $E(G[V(S_{3,2})])\setminus E(S_{3,2})$ has no common vertex. 
\end{lem}
\begin{proof}
  The $Q$-matrix of $S_{3,2}$ has the following form:
  \[
  Q(G)|_{V(S_{3,2})}=
  \begin{pmatrix}
    4 & 1 & 1          & 1          & 1          & 0          & 0         \\
    1 & 3 & 0          & 0          & 0          & 1          & 1         \\
    1 & 0 & d(x_0)     & a_{x_0x_1} & a_{x_0x_2} & a_{x_0y_0} & a_{x_0y_1}\\
    1 & 0 & a_{x_1x_0} & d(x_1)     & a_{x_1x_2} & a_{x_1y_0} & a_{x_1y_1}\\
    1 & 0 & a_{x_2x_0} & a_{x_2x_1} & d(x_2)     & a_{x_2y_0} & a_{x_2y_1}\\
    0 & 1 & a_{y_0x_0} & a_{y_0x_1} & a_{y_0x_2} & d(y_0)     & a_{y_0y_1}\\
    0 & 1 & a_{y_1x_0} & a_{y_1x_1} & a_{y_1x_2} & a_{y_1y_0} & d(y_1)
  \end{pmatrix},
  \]
  where $d(x_0), d(x_1), d(x_2)\in\{1,2,3\}$ and $d(y_0), d(y_1)\in\{1,2,3,4\}$
  and $a_{\cdot\cdot}\in\{0,1\}$.
  Suppose that there exist two edges in $E(G[V(S_{3,2})])\setminus E(S_{3,2})$ sharing a vertex $v$.
  Note that two endvertices, not $v$, of the two edges are in $\{x_0$, $x_1$, $x_2$, $y_0$, $y_1\}$.
  If these two endvertices are in $\{x_0,x_1,x_2\}$ or $\{y_0,y_1\}$, 
  then the largest eigenvalue is greater than $6$ or the smallest eigenvalue is less than $1$. 
  Thus we may assume that each of these two endvertices is in $\{x_0,x_1,x_2\}$ and $\{y_0,y_1\}$, respectively.
  \begin{itemize}
    \item[(a)] $v\in\{x_0,x_1,x_2\}$. 
  \end{itemize} 
  By the symmetry, we may assume that $v=x_0$ and $v$ is adjacent to $x_1$.
  Then $d(x_2)\in\{2,3\}$ and $x_2$ is not adjacent to any vertex in $V(S_{3,2})\setminus\{x\}$ by Proposition~\ref{prop:ev}.
  Thus, there exists $z_1$ in $V(G)\setminus V(S_{3,2})$ such that $z_1$ is adjacent to $x_2$. Then the $Q$-matrix of $G[V(S_{3,2})\cup\{z_1\}]$ has the following form:

  {\footnotesize \[
      \left(
      \begin{array}{ccccccc|c}
    4 & 1 & 1          & 1          & 1        & 0          & 0          & 0        \\
    1 & 3 & 0          & 0          & 0        & 1          & 1          & 0        \\
    1 & 0 & d(x_0)     & 1          & 0        & a_{x_0y_0} & a_{x_0y_1} & a_{x_0z_1} \\
    1 & 0 & 1          & d(x_1)     & 0        & a_{x_1y_0} & a_{x_1y_1} & a_{x_0z_1} \\ 
    1 & 0 & 0          & 0          & d(x_2)   & 0          & 0          & 1          \\
    0 & 1 & a_{y_0x_0} & a_{y_0x_1} & 0        & d(y_0)     & a_{y_0y_1} & a_{y_0z_1} \\ 
    0 & 1 & a_{y_1x_0} & a_{y_1z_1} & 0        & a_{y_1y_0} & d(y_1)     & a_{y_1z_1} \\ \hline
    0 & 0 & a_{z_1x_0} & a_{z_1x_1} & 1        & a_{z_1y_0} & a_{z_1y_1} & d(z_1)        
      \end{array}
      \right),
  \]}where $a_{x_0y_0} + a_{x_0y_1}=1$.
  For every $d(\cdot)\in\{1,2,3,4\}$ and $a_{\cdot\cdot}\in\{0,1\}$, 
  the $Q$-matrix does not satisfy Proposition~\ref{prop:ev}.

  \begin{itemize}
    \item[(b)] $v\in\{y_0,y_1\}$. 
  \end{itemize} 
  By the symmetry, we may assume that $v$ is adjacent to $x_0$.
  Then $d(x_0)=3$ and $x_0$ is not adjacent to any vertex in $V(S_{3,2})\setminus\{x,v\}$ by Proposition~\ref{prop:ev}.
  Thus, there exists $z_1$ in $V(G)\setminus V(S_{3,2})$ such that $z_1$ is adjacent to $x_0$. Then the $Q$-matrix of $G[V(S_{3,2})\cup\{z_1\}]$ has the following form:
  {\footnotesize \[
      \left(
      \begin{array}{ccccccc|c}
        4 & 1 & 1          & 1          & 1          & 0          & 0          & 0        \\
        1 & 3 & 0          & 0          & 0          & 1          & 1          & 0        \\
        1 & 0 & 3          & 0          & 0          & a_{x_0y_0} & a_{x_0y_1} & 1         \\
        1 & 0 & 0          & d(x_1)     & a_{x_1x_2} & a_{x_1y_0} & a_{x_1y_1} & a_{x_0z_1} \\ 
        1 & 0 & 0          & a_{x_2x_1} & d(x_2)     & a_{x_2y_0} & a_{x_2y_1} & a_{x_2z_1} \\
        0 & 1 & a_{y_0x_0} & a_{y_0x_1} & a_{y_0x_2} & d(y_0)     & 1          & a_{y_0z_1} \\ 
        0 & 1 & a_{y_1x_0} & a_{y_1x_1} & a_{y_1x_2} & 1          & d(y_1)     & a_{y_1z_1} \\ \hline
        0 & 0 & 1          & a_{z_1x_1} & a_{z_1x_2} & a_{z_1y_0} & a_{z_1y_1} & d(z_1)           
      \end{array}
      \right),
  \]}where $a_{x_0y_0} + a_{x_0y_1}=1$.
  For every $d(\cdot)\in\{1,2,3,4\}$ and $a_{\cdot\cdot}\in\{0,1\}$, 
  the $Q$-matrix does not satisfy Proposition~\ref{prop:ev}.
\end{proof}

\begin{lem}\label{lem:s32edgeleq2} 
  Let $G$ be a connected non-bipartite $Q$-integral graph with $Q$-spectral radius $6$ and maximum edge-degree $5$.
  Suppose that $G$ contains $S_{3,2}$ as a subgraph. 
  Then $|E(G[V(S_{3,2})])\setminus E(S_{3,2})|\leq2$.
\end{lem}
\begin{proof}
  Every edge in $E(G[V(S_{3,2})])\setminus E(S_{3,2})$ is not incident to $x$ and $y$.
  By Lemma~\ref{lem:s32disjoint}, for any two edges in $|E(G[V(S_{3,2})])\setminus E(S_{3,2})|$, they do not share a vertex. Thus, $E(G[V(S_{3,2})])\setminus E(S_{3,2})$ has at most two edges.
\end{proof}

By Lemmas~\ref{lem:s32disjoint} and~\ref{lem:s32edgeleq2}, we consider the number of edges in $E(G[V(S_{3,2})])\setminus E(S_{3,2})$.

\begin{lem}\label{lem:s32not2}
  Let $G$ be a connected non-bipartite $Q$-integral graph with $Q$-spectral radius $6$ and maximum edge-degree $5$. Suppose that $G$ contains $S_{3,2}$ as a subgraph. Then $|E(G[V(S_{3,2})])\setminus E(S_{3,2})| \neq 2$.
\end{lem}
\begin{proof}
  By Lemma~\ref{lem:s32disjoint}, there are exactly two edges that do not share a vertex. 
Then one of the following holds:
\begin{itemize}
  \item[(a)] $a_{x_0y_0}=a_{x_1y_1}=a_{y_0x_0}=a_{y_1x_1}=1$ and $a_{\cdot\cdot}=0$;
  \item[(b)] $a_{x_0y_0}=a_{x_1x_2}=a_{y_0x_0}=a_{x_2x_1}=1$ and $a_{\cdot\cdot}=0$;
  \item[(c)] $a_{x_0x_1}=a_{y_0y_1}=a_{x_1x_0}=a_{y_1y_0}=1$ and $a_{\cdot\cdot}=0$;
\end{itemize} 
  To prove cases (a) and (b), assume that $x_0y_0\in E(G[V(S_{3,2})])\setminus E(S_{3,2})$ and $a_{x_1y_1}+a_{x_1x_2}=1$.
  By the symmetry, $a_{y_1x_1}+a_{x_2x_1}=1$.
  By Proposition~\ref{prop:ev}, we have $d(x_0)=3$.
  Thus, there exists $z_1$ in $V(G)\setminus V(S_{3,2})$ such that $z_1$ is adjacent to $x_0$.
  Then the $Q$-matrix of $G[V(S_{3,2})\cup\{z_1\}]$ has the following form:
  
  {\footnotesize \[
    \left(
    \begin{array}{ccccccc|c}
      4 & 1 & 1      & 1          & 1          & 0          & 0          & 0      \\
      1 & 3 & 0      & 0          & 0          & 1          & 1          & 0      \\
      1 & 0 & 3 & 0          & 0          & 1          & 0          & 1      \\
      1 & 0 & 0      & d(x_1)     & a_{x_1x_2} & 0          & a_{x_1y_1} & a_{x_1z_1}\\
      1 & 0 & 0      & a_{x_2x_1} & d(x_2)     & 0          & 0          & a_{x_2z_1}\\
      0 & 1 & 1      & 0          & 0          & d(y_0)     & 0          & a_{y_0z_1}\\
      0 & 1 & 0      & a_{y_1x_1} & 0          & 0          & d(y_1)     & a_{y_1z_1}\\ \hline
      0 & 0 & 1      & a_{z_1x_1} & a_{z_1x_2} & a_{z_1y_0} & a_{z_1y_1}          & d(z_1)        
    \end{array}
    \right),
  \]}
  where $a_{x_1y_1}+a_{x_1x_2}=1$.
  If $a_{x_1y_1}=1$, then the $Q$-matrix does not satisfy Proposition~\ref{prop:ev}.
  If $a_{x_1x_2}=1$, then the largest eigenvalue of the $Q$-matrix is equal to 6, so we obtain the whole graph $G$, but it is not $Q$-integral, a contradiction.
 By applying the method the case (c) in the same way as in Lemma~\ref{lem:t32x1y0}, we can conclusively show that no $Q$-integral graph exists.
\end{proof}

\begin{lem}\label{lem:s32not1}
  Let $G$ be a connected non-bipartite $Q$-integral graph with $Q$-spectral radius $6$ and maximum edge-degree $5$.
  Suppose that $G$ contains $S_{3,2}$ as a subgraph. 
  Then $|E(G[V(S_{3,2})])\setminus E(S_{3,2})| \neq 1$.
\end{lem}
\begin{proof}
  By Lemma~\ref{lem:s32disjoint}, there is exactly one edge. Then one of the following holds:
  \begin{itemize}
    \item[(a)] $a_{x_0y_0}=1$ and $a_{\cdot\cdot}=0$;
    \item[(b)] $a_{x_0x_1}=1$ and $a_{\cdot\cdot}=0$;
    \item[(c)] $a_{y_0y_1}=1$ and $a_{\cdot\cdot}=0$.
  \end{itemize} 
  To prove case (a), assume that $\{x_0y_0\}=E(G[V(S_{3,2})])\setminus E(S_{3,2})$.
  By the symmetry $a_{y_0x_0}=1$.
  By Proposition~\ref{prop:ev}, we have $d(x_0)=3$ and $d(x_1)=3$.
  Thus, there exists $z_1$ in $V(G)\setminus V(S_{3,2})$ such that $z_1$ is adjacent to $x_0$. 
  In this case $a_{z_1u}=0$ for $u\in\{x_1,x_2,y_0\}$.
  Since $d(x_1)=3$, 
  there exists $z_2$ in $V(G)\setminus V(S_{3,2})$ such that $z_2$ is adjacent to $x_1$.
  Then the $Q$-matrix of $G[V(S_{3,2})\cup\{z_1,z_2\}]$ has the following form:
  \[
  \scalemath{0.9}{
    \left(
    \begin{array}{ccccccc|cc}
      4 & 1 & 1      & 1          & 1          & 0          & 0          & 0    &0  \\
      1 & 3 & 0      & 0          & 0          & 1          & 1          & 0    &0  \\
      1 & 0 & d(x_0) & 0          & 0          & 1          & 0          & 1    &a_{x_0,z_2}  \\
      1 & 0 & 0      & d(x_1)     & 0          & 0         & 0 & 0              &1 \\
      1 & 0 & 0      & 0 & d(x_2) & 0          & 0          & 0                 &a_{x_2,z_2}\\
      0 & 1 & 1      & 0          & 0          & d(y_0)     & 0          & 0    &a_{x_3,z_2}\\
      0 & 1 & 0      & 0          & 0          & 0          & d(y_1)     & a_{y_1z_1} & a_{y_1, z_2}\\ \hline
      0 & 0 & 1      & 0 & 0 & 0 & a_{z_1y_1}          & d(z_1)                       & a_{z_1,z_2}  \\
      0 & 0 & a_{z_2,x_0}      & 1 & a_{z_2,x_2} & a_{z_2,y_0} & a_{z_2,y_1}          & a_{z_2,z_1} & d(z_2)         
    \end{array}
    \right)}.
  \]
  For every $d(\cdot)\in\{1,2,3,4\}$ and $a_{\cdot\cdot}\in\{0,1\}$, 
  the $Q$-matrix does not satisfy Proposition~\ref{prop:ev}.
  By applying the method to cases (b) and (c) in the same way as in Lemma~\ref{lem:t32x1y0}, we can conclusively show that no $Q$-integral graph exists.
\end{proof}

\begin{lem}\label{lem:s32not0}
  Let $G$ be a connected non-bipartite $Q$-integral graph with $Q$-spectral radius $6$ and maximum edge-degree $5$.
  Suppose that $G$ contains $S_{3,2}$ as a subgraph. 
  Then $E(G[V(S_{3,2})])\setminus E(S_{3,2})\not= \emptyset$. 
\end{lem}
\begin{proof}
 We assume that $|E(G[V(S_{3,2})])\setminus E(S_{3,2})|=0$. 
   Similarly to Lemma~\ref{lem:t32x1y0}, we can show that there is no $Q$-integral graph in this case.
\end{proof}


\subsection{Main Results}\label{sec:proof}

\noindent \textbf{\textit{Proof of Theorem~\ref{thm:fish}}}.
Suppose there exist two adjacent vertices $x$ with degree $4$ and $y$ with degree $3$ in $G$.
By Lemma~\ref{2common}, $x$ and $y$ have at most one common neighbor.
If there is a vertex that is adjacent to both $x$ and $y$, then there is only one graph by Lemma~\ref{lem:t32_2edge}.
If there is no common neighbor of $x$ and $y$, there is no graph, by Lemmas~\ref{lem:s32edgeleq2},~\ref{lem:s32not2},~\ref{lem:s32not1} and~\ref{lem:s32not0}.
This completes the proof. 
\hfill \textsquare

\vspace{3mm}

\noindent \textbf{\textit{Proof of Theorem~\ref{thm:main}}}. 
Let $G$ be a connected non-bipartite $Q$-integral graph with $Q$-spectral radius at most $5$. 
Then $G$ is isomorphic to one of the graphs $G_1$ and $G_2$ in Figure~\ref{fig:main} (see Table~\ref{tab:list}).
Let us assume that $G$ has $Q$-spectral radius $6$.
Then $G$ is isomorphic to one of the 17 graphs in Theorem~\ref{thm:PS19} (i), or $G$ has maximum edge-degree $5$.
Among the 17 graphs, five are non-bipartite: $G_3$, $G_4$, $G_5$, $G_6$, and $G_7$ in Figure~\ref{fig:main}.
Now, we consider $G$ with a maximum edge-degree $5$.
Then there exists an edge $xy$ of $G$ whose edge-degree is $5$.

By Proposition~\ref{prop:verdeg}, the maximum vertex degree of $G$ is at most $4$.
Hence, the vertex degrees of $x$ and $y$ are $4$ and $3$.
Then $G$ is the striped fish graph by Theorem~\ref{thm:fish}.
This completes the proof. 
\hfill \textsquare

\section*{Appendix}

The proof process can be structured into the following pseudocode.
To confirm the truth of the results proven at each step of Theorem~\ref{thm:fish}, we converted this pseudocode into actual code, executed it, and verified the results.
We note that it can also be used for non-bipartite $Q$-integral graphs with $Q$-spectral radius greater than $6$.

\begin{algorithm}[!ht] \label{alg:GraphExtension}
    \LinesNumbered
    \caption{{\bf Algorithm}}
    \KwInput{
        $H$ : a graph
    }
    \KwOutput{
        $\mathcal{H}$ : a list of graphs
        }
    \Begin{
        $\mathcal{H}$ $\leftarrow$ an empty list

        \texttt{d\_list} $\leftarrow$ an empty list

        \For{
            $d:V(H) \rightarrow \{1,2, \ldots, \rho-2\}$
        }{
            \If{\rm
                $(H,d)$ satisfies Proposition~\ref{prop:ev}
            }{
                Add $d$ to \texttt{d\_list}
            }
        }

        Take $\tilde{v}$ not in $V(H)$

        \For{\rm
                a connected graph $\tilde{H}$ by adding $\tilde{v}$ to $H$
            }{
                \If{\rm
                    $(\tilde{H},\tilde{d})$ satisfies Proposition~\ref{prop:ev} for some $\tilde{d}: V(\tilde{H}) \rightarrow \{1,2, \ldots, \rho-2\}$ 
                }{
                    \uIf{\rm 
                        $\exists v \in V(H)$ such that $\deg_{H}(v) < \min_{d \in \mathtt{d\_list}} d(v)$
                    }{
                        \If{\rm
                            $v$ is adjacent to $\tilde{v}$ in $\tilde{H}$
                        }{
                            Add $\tilde{H}$ to $\mathcal{H}$
                        }                        
                    }
                    \Else{
                        Add $\tilde{H}$ to $\mathcal{H}$
                    }
                }
            }
        \Return{$\mathcal{H}$}
    }
\end{algorithm}

\section*{Acknowledgments}
%
Jeong Rye Park was supported by the NRF of Korea grant funded by the Korea government (NRF-2018R1D1A1B07048197 and NRF-2021R1I1A1A01057767).
Jongyook Park was supported by the NRF of Korea grant funded by the Korea government (MSIT) (NRF-2020R1A2C1A01101838).
Yoshio Sano was supported by JSPS KAKENHI Grant Number JP19K03598.
Semin Oh was supported by the National Research Foundation (NRF) and the Ministry of Trade, Industry and Energy (MOTIE) of Korea grant funded by the Korea government (NRF-2020R1I1A1A01073818
 and 
 MOTIE-P142700014).

\section*{Declarations}

\textbf{Conflict of interest} The authors declare no conflict of interest.

%

\printbibliography 

\end{document}